\documentclass[11pt]{article}
\usepackage{graphicx,algorithmic,algorithm}%
\usepackage{enumerate}
\usepackage{tikz}
\usetikzlibrary{shapes}
\usepackage{hyperref}
\hypersetup{
    pdftitle=   {Graphs of bounded depth-$2$ rank-brittleness},
   pdfauthor=  {O-joung Kwon and Sang-il Oum}
}
\usepackage{todonotes}
\usepackage{authblk}
\usepackage{amsthm,amssymb,amsmath}
\usepackage{mathabx}
\newtheorem{theorem}{Theorem}[section]
\newtheorem{lemma}[theorem]{Lemma}
\newtheorem{proposition}[theorem]{Proposition}
\newtheorem{corollary}[theorem]{Corollary}
\newtheorem{conjecture}[theorem]{Conjecture}

\newtheorem*{thm:main}{Theorem~\ref{thm:main}}
\newtheorem*{corollarylrw}{Corollary~\ref{cor:lrw}}

\newcommand\abs[1]{\lvert #1\rvert}

\newcommand\mat{\boxminus}

\newcommand\tri{\boxslash}
\newcommand\antimat{\boxtimes}

\newcommand\rkbrit{\beta^\cutrk}

\def\K_#1{{K_{#1}}}
\def\S_#1{\overline{K_{#1}}}

\newcommand\cutrk{\rho}

\newcommand\lrw{\operatorname{lrw}}
\newcommand\rd{\operatorname{rd}}
\newcommand\rbrit{\operatorname{rbrit}}

\newcommand{\cF}{\mathcal{F}}

\begin{document}
\title{Graphs of bounded depth-$2$ rank-brittleness}

\author[1,2]{O-joung Kwon\thanks{Supported by the National Research Foundation of Korea (NRF) grant funded by the Ministry of Education (No. NRF-2018R1D1A1B07050294).}}
\affil[1]{Department of Mathematics, Incheon National University, Incheon,~Korea}
\author[2,3]{Sang-il Oum\thanks{Supported by IBS-R029-C1.}}
\affil[2]{Discrete Mathematics Group,
  Institute for Basic Science (IBS), Daejeon,~Korea}
\affil[3]{Department of Mathematical Sciences, KAIST, Daejeon,~Korea}
\affil[ ]{\small \textit{Email addresses:}  
  \texttt{ojoungkwon@gmail.com}, \texttt{sangil@ibs.re.kr}}
\date\today
\maketitle

\begin{abstract}
  We characterize classes of graphs
  closed under taking vertex-minors and 
  having no $P_n$ and no disjoint union of $n$ copies of the $1$-subdivision of $K_{1,n}$ for some $n$.
  Our characterization is described in terms of a tree of radius $2$ whose leaves are labelled by the vertices of  a graph $G$,
  and the width is measured by the maximum possible cut-rank of a partition of $V(G)$ induced by splitting an internal node of the tree to make two components. The minimum width possible is called the depth-$2$ rank-brittleness of $G$.
  We prove that
  for all $n$, every graph with sufficiently large depth-$2$ rank-brittleness
  contains $P_n$ or disjoint union of $n$ copies of the $1$-subdivision of $K_{1,n}$
  as a vertex-minor.
\end{abstract}

\section{Introduction}

\emph{Tree-depth} is a graph parameter in the theory of sparse graph classes, which measures how far a graph is from being a star, introduced by Ne{\v{s}}et{\v{r}}il and Ossona de Mendez~\cite{nevsetvril2006tree}. An equivalent concept has been introduced a few times under the names like the vertex ranking number and the minimum height elemination tree~\cite{bodlaender1998rankings,deogun1994vertex,schaffer1989optimal}. 
It is known that a graph has large tree-depth if and only if it has a long path, see \cite[Section 6.2]{NO2012}.

For some applications, it is desirable to say that complete graphs are also very similar to stars. However, complete graphs have unbounded tree-depth.
To design a graph parameter similar to tree-depth but more suitable for dense graph classes, DeVos, Kwon, and Oum~\cite{DKO2019} introduced the \emph{rank-depth} of a graph.
Roughly speaking, the rank-depth of a graph $G$ is defined in terms of 
a \emph{decomposition},
which is a tree whose leaves are labelled by the vertices of $G$.
A decomposition has two qualities, one of which is the radius of the tree,
and the other is the maximum width of internal nodes, measured
by some connectivity function of $G$.
The rank-depth of a graph $G$ is defined as the minimum integer $k$
such that $G$ admits a decomposition of radius at most $k$ and width at most $k$.
The detailed definition of rank-depth will be reviewed in Section~\ref{sec:prelim}.
In fact, there was an equivalent concept called the \emph{shrub-depth} of classes of graphs, introduced by Ganian, Hlin\v{e}n\'y, Ne{\v{s}}et{\v{r}}il, Obdr\v{z}\'alek, Ossona de Mendez, and Ramadurai~\cite{GHNOO2017,GHNOOR2012}.
The definition of shrub-depth uses logical terms similar to the definition of clique-width~\cite{CO2000}, while the definition of rank-depth uses a tree-like decomposition similar to that of rank-width~\cite{OS2004}.
DeVos, Kwon, and Oum~\cite{DKO2019} showed that a class of graphs has bounded rank-depth if and only if it has bounded shrub-depth.

Hlin{\v{e}}n{\'y}, Kwon, Obdr{\v{z}}{\'a}lek, and Ordyniak~\cite{HlinenyKJS2016}
proposed the following conjecture, which we state in terms of rank-depth.
To state their conjecture, we first introduce vertex-minors. 
The \emph{local complementation} at a vertex $v$ of a graph $G$ is an operation
to obtain a new graph $G*v$ from $G$ by removing all edges $xy$ between two adjacent pairs $x$, $y$ of neighbors of $v$ and adding edges $xy$ for all non-adjacent pairs $x$, $y$ of neighbors of $v$. 
A graph $H$ is a \emph{vertex-minor} of a graph $G$ if $H$ can be obtained from $G$ by a sequence of local complementations and vertex deletions. 
It is known that the rank-depth of a vertex-minor of $G$ is at most the rank-depth of $G$ and so it is natural to think of an obstruction for graphs of bounded rank-depth in terms of vertex-minors. The following conjecture states that paths are obstructions for having bounded rank-depth. 
This conjecture was verified for graphs of rank-width $1$ by Novotn{\'y}~\cite[Theorem 6.3.2]{Novotny2016}.

\begin{conjecture}[Hlin{\v{e}}n{\'y}, Kwon, Obdr{\v{z}}{\'a}lek, and Ordyniak~\cite{HlinenyKJS2016}]\label{conj:rankdepth}
 A class~$\mathcal{C}$ of graphs has bounded rank-depth if and only if there exists an integer $t$ such that no graph $G\in \mathcal{C}$ contains a path of length $t$ as a vertex-minor.
 \end{conjecture}

As a step towards Conjecture~\ref{conj:rankdepth}, we define a new parameter called \emph{depth-$d$ rank-brittleness} for an integer $d$
by restricting the radius of the tree in the decomposition to be at most $d$ in the
definition of rank-depth.
The \emph{depth-$d$ rank-brittleness} of a graph $G$ is the minimum integer $k$ such that 
$G$ admits a decomposition of radius at most $d$ and width at most $k$.
We denote this parameter by $\rbrit_d(G)$. By definition,
the rank-depth of a graph $G$ is at most $\max\{d, \rbrit_d(G)\}$ for all $d\ge 1$ and 
\[\rbrit_1(G)\ge \rbrit_2(G)\ge \rbrit_3(G)\ge \cdots .\]
In Section~\ref{sec:application}, we will show that
a graph of rank-depth $k$ has linear rank-width at most $k^2$.

        A class $\mathcal{C}$ of graphs is a \emph{vertex-minor ideal} if 
	for every graph $G\in \mathcal{C}$,
        $\mathcal C$ contains all graphs isomorphic to vertex-minors of $G$.
	For a graph $H$, we write $nH$ for the disjoint union of $n$ copies of $H$. 
  It is straightforward to deduce the following proposition by using Ramsey-type results. To see this, one can use Theorem~\ref{thm:largebipartite}, Ramsey's theorem, and Lemma~\ref{lem:antimattopath}. It can be also seen as a special case of a theorem due to Kwon and Oum~\cite[Theorem 1.4]{KO2018sca}, which is stated in Theorem~\ref{thm:scattered}.

        \begin{proposition}%
          A vertex-minor ideal $\mathcal C$ 
          has bounded depth-$1$ rank-brittleness if and only if
          $\{K_2,2K_2,3K_2,\ldots\}\not\subseteq \mathcal C$.
        \end{proposition}
		
	In this paper, we characterize classes of graphs of bounded depth-$2$ rank-brittleness in terms of forbidden vertex-minors.
	Let $T_{2,n}$ be the $1$-subdivision of $K_{1,n}$, see Figure~\ref{fig:t25}.
        Here is our main theorem.

        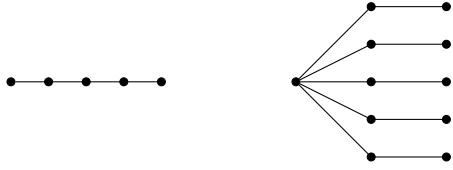
\begin{figure}
          \centering
                  \begin{tikzpicture}[baseline]
            \tikzstyle{v}=[circle, draw, solid, fill=black, inner sep=0pt, minimum width=3pt]
            \foreach \y in {0,1,2,3,4}{
              \node [v] at (.5*\y,0) (v\y) {};
            }
            \draw (v0)--(v4);
          \end{tikzpicture}
          \qquad\qquad
          \begin{tikzpicture}[baseline]
            \tikzstyle{v}=[circle, draw, solid, fill=black, inner sep=0pt, minimum width=3pt]
            \node [v] at (0,0) (v) {};
            \foreach \y in {-.5,-1,0,1,.5} {
              \draw (v)--(1,\y) node[v]{}--(2,\y) node[v]{};
            }
          \end{tikzpicture}
          \caption{Graphs $P_5$ and $T_{2,5}$.}
          \label{fig:t25}
        \end{figure}
        
		\begin{theorem}\label{thm:main}
		A vertex-minor ideal $\mathcal{C}$ has bounded depth-$2$ rank-brittleness if and only if 
		\[ \{P_1, P_2, P_3, P_4, \ldots \}\nsubseteq \mathcal{C}
                \text{ and }
		 \{T_{2,1}, 2T_{2,2}, 3T_{2,3}, 4T_{2,4}, \ldots \}\nsubseteq \mathcal{C}. \]
       		\end{theorem}

                Since  $T_{2,n}$ contains $P_5$ if $n\ge 2$,
                we obtain the following corollary,
                confirming a weaker statement of Conjecture~\ref{conj:rankdepth}.
	
	\begin{corollary}\label{cor:lrw}
          For every positive integer $n$,
          graphs with no vertex-minors isomorphic to $nP_5$
          have bounded depth-$2$ rank-brittleness, bounded rank-depth, and bounded linear rank-width.
	\end{corollary}
	
	We sketch the proof of Theorem~\ref{thm:main}.
	It is straightforward to show that $P_n$ and $nT_{2,n}$ have large depth-$2$ rank-brittleness.
	We mainly show that for every fixed $n$,
	if a graph $G$ has sufficiently large depth-$2$ rank-brittleness, 
	then it has a vertex-minor isomorphic to  $P_n$ or $nT_{2,n}$. 
        A theorem of Kwon and Oum~\cite[Theorem 1.4]{KO2018sca}
        will imply that every graph of large depth-$2$ rank-brittleness
        has a vertex-minor isomorphic to $aK_b$
        for large $a$ and $b$.
        By taking a graph locally equivalent to $G$, we may assume that
        $G$ has an induced subgraph isomorphic to $aK_b$.
	
	In Section~\ref{sec:rank1pair}, we prove that if a graph $H$ contains $3$
        pairwise twins, then
        one of them can be removed without decreasing the depth-$2$ rank-brittleness.
        Using that, each component $C$ of $aK_b$ can be partitioned into
        at least $b/2$ sets such that vertices in distinct sets are not twins.
	By the Ramsey-type result on bipartite graphs,
        we will extract a large (induced) matching or an anti-matching or a half graph between 
	$C$ and the rest. We find this for each component of $aK_b$.
        Then using the sunflower lemma and Ramsey's theorem,
        we will clean up all the structures and find a vertex-minor isomorphic to $nT_{2,n}$ or $P_n$.
	Section~\ref{sec:construction} is devoted to describe all the intermediate structures.
	The proof of Theorem~\ref{thm:main} is given in Section~\ref{sec:mainproof}.
        Section~\ref{sec:application} shows an inequality between linear rank-width and rank-depth and presents a corollary of Theorem~\ref{thm:main} for graphs with no vertex-minors isomorphic to $nP_5$.
			
		\section{Preliminaries}\label{sec:prelim}
		
		All graphs in this paper are simple and undirected. 
		For a graph $G$, we denote by $V(G)$ and $E(G)$ the vertex set and the edge set of $G$, respectively. 
		Let $G$ be a graph.
	For $S\subseteq V(G)$, we denote by $G[S]$ the subgraph of $G$ induced by $S$, and  
	for two disjoint vertex subsets $S$ and $T$ of $G$, 
	we denote by $G[S,T]$ the bipartite graph with bipartition $(S,T)$ such that 
	for $a\in S$ and $b\in T$, $a,b$ are adjacent in $G[S,T]$ if and only if they are adjacent in $G$.
  For $v\in V(G)$, we denote by $G- v$ the graph obtained from $G$ by removing~$v$ and all edges incident with $v$.
  For a set $X$ of vertices, we denote by $G-X$ the graph obtained from $G$ by deleting all vertices in $X$ and all edges incident with those vertices.
	For $v\in V(G)$, the set of \emph{neighbors} of $v$ in $G$ is denoted by $N_G(v)$, and the \emph{degree} of $v$ is the size of $N_G(v)$.
	We denote by $A(G)$ the \emph{adjacency matrix} of $G$.
	
	For two disjoint vertex subsets $A$ and $B$ of $G$, 
	we say that $A$ is \emph{complete} to $B$ if
        every vertex in $A$ is adjacent to all vertices in $B$.
	Similarly, $A$ is \emph{anti-complete} to $B$, if
        every vertex in $A$ is non-adjacent to all vertices in $B$.
        A \emph{clique} is a set of pairwise adjacent vertices
        and an \emph{independent} set is a set of pairwise non-adjacent vertices.
        
		Two vertices $v$ and $w$ in a graph $G$ are called \emph{twins} if 
		$N_G(v)\setminus \{v,w\}=N_G(w)\setminus \{v,w\}$.
		Note that a set of pairwise twins is either a clique or an independent set.

	Let $K_n$ denote the complete graph on $n$ vertices, and let $K_{1,n}$ denote the star with $n$ leaves.
	Let $P_n$ denote the path on $n$ vertices. 
	For a graph $G$, we denote by $\overline{G}$ the \emph{complement} of $G$, that is, two vertices $v$ and $w$ in $G$ are adjacent if and only if they are not adjacent in $\overline{G}$.

		We write $R(n;k)$ to denote the minimum number $N$ such
	that every coloring of the edges of $K_N$ into $k$ colors induces a monochromatic complete subgraph on $n$ vertices. The classical theorem of Ramsey implies that $R(n;k)$ exists.

	We also use the sunflower lemma. Let $\cF$ be a family of sets. 
	A subset $\{M_1, M_2, \ldots, M_p\}$ of $\cF$ is a \emph{sunflower} with \emph{core} $A$ (possibly an empty set) and $p$ \emph{petals} if
	for all distinct $i,j\in \{1,2, \ldots, p\}$, $M_i\cap M_j=A$. 
	
	\begin{theorem}[Sunflower Lemma~{\cite[Erd\H{o}s and Rado]{ErdosR1960}}]\label{thm:sunflower}
	Let $k$ and $p$ be positive integers, and 
	$\cF$ be a family of sets each of cardinality $k$. 
	If $\abs{\cF}> k!(p-1)^k$, then $\cF$ contains a sunflower with $p$ petals.
      \end{theorem}

      \subsection{Vertex-minors}
For a vertex $v$ in a graph $G$, to perform \emph{local complementation} at~$v$, 
replace the subgraph of $G$ induced on $N_G(v)$ 
by its complement graph. 
We write $G*v$ to denote the graph obtained from $G$ by applying
local complementation at~$v$.
Two graphs $G$ and $H$ are \emph{locally equivalent} if $G$ can be obtained from $H$ by a sequence of local complementations.
A graph $H$ is a \emph{vertex-minor} of a graph $G$
if $H$ is an induced subgraph of a graph which is locally equivalent to $G$.

        \subsection{Rank-depth and rank-brittleness}
The \emph{cut-rank} function of a graph $G$, denoted by $\cutrk_G(S)$ for a subset $S$ of $V(G)$,
is defined as the rank of an $S\times (V(G)\setminus S)$ $0$-$1$ matrix over the binary field whose $(a,b)$-entry for $a\in S$, $b\notin S$ is $1$ if $a$, $b$ are adjacent and $0$ otherwise.
The cut-rank function is invariant under the local complementation, see Oum~\cite{Oum2004}.
The cut-rank function satisfies the \emph{submodular inequality}, that is, for all $X,Y\subseteq V(G)$, 
$\rho_G(X)+\rho_G(Y)\ge \rho_G(X\cap Y)+\rho_G(X\cup Y)$. 
The \emph{$\rho_G$-width} of a partition $\mathcal P=(X_1,X_2,\ldots,X_m)$ of $V(G)$, for some $m$,
        is \[
          \max\left\{ \rho_G\bigl(\bigcup_{i\in I} X_i\bigr): I\subseteq \{1,2,\ldots,m\}\right\}.
        \]

A \emph{decomposition} of a  graph $G$  is 
a pair $(T,\sigma)$ of a tree $T$ with at least one internal node and a bijection $\sigma$ from $V(G)$ to the set of leaves of $T$.
The \emph{radius} of a decomposition $(T,\sigma)$ 
is defined to be the radius of the tree~$T$. 
For an internal node $v \in V(T)$, the components of the graph $T - v$ give rise to a partition $\mathcal P_v$ of $V(G)$ by $\sigma$
and the \emph{width} of $v$ is defined to be 
the $\rho_G$-width of $\mathcal P_v$. 
The \emph{width} of the decomposition $(T,\sigma)$ is the maximum width of an internal node of $T$.
We say that a decomposition $(T, \sigma)$ is a $(k,r)$-\emph{decomposition} of $G$
if the width is at most $k$ and the radius is at most $r$. 
The \emph{rank-depth} of a graph $G$, denoted by $\rd(G)$, is the minimum integer $k$ such that 
$G$ admits a $(k, k)$-decomposition.
If $\abs{V(G)}<2$, then there exists no decomposition and rank-depth is defined to be $0$.
Note that every tree in a decomposition has radius at least $1$ and therefore
the rank-depth of a graph is at least $1$ if $\abs{V(G)}\ge 2$.

The \emph{depth-$d$ rank-brittleness} of a graph $G$,
denoted by $\rbrit_d(G)$, is the minimum integer $k$
such that $G$ admits a $(k,d)$-decomposition.
If $\abs{V(G)}<2$, then we define $\rbrit_d(G)=0$.
Note that the depth-$1$ rank-brittleness of a graph $G$
is equal to $\max_{A\subseteq V(G)} \rho_G(A)$.

\subsection{Constructions of common graphs}
For two graphs $G$ and $H$ on the disjoint vertex sets, each having $n$ vertices, we would like
to introduce operations to construct  graphs on $2n$ vertices by making the disjoint
union of them and adding some edges between two graphs. 
Roughly speaking, $G\mat H$ will add a perfect matching,
 $G\antimat H$ will add the complement of a perfect matching, 
 and $G\tri H$ will add a half graph.
Formally, 
for two $n$-vertex graphs $G$ and $H$ with fixed ordering on the vertex sets $\{v_1,v_2,\ldots,v_n\}$ and $\{w_1, w_2, \ldots, w_n\}$ respectively, 
let 
$G\mat H$, $G\antimat H$, $G\tri H$
be graphs on the vertex set $V(G)\cup V(H)$ 
whose subgraph induced by $V(G)$ or $V(H)$
is $G$ or $H$, respectively
such that
 for all $i,j\in \{1,2,\ldots,n\}$, 
\begin{enumerate}[(i)]
\item 
$v_iw_j\in E(G\mat H)$ if and only if $i=j$,
\item 
$v_iw_j\in E(G\antimat H)$ if and only if $i\neq j$,
\item 
$v_iw_j\in E(G\tri H)$ if and only if $i\ge j$.
\end{enumerate}
See Figure~\ref{fig:construction} for illustrations of $\K_5\mat\S_5$, $\K_5\antimat\S_5$, and $\K_5\tri \S_5$.

\begin{figure}
 \tikzstyle{v}=[circle, draw, solid, fill=black, inner sep=0pt, minimum width=3pt]
  \centering
  \newcommand\Sfive[1]{    \begin{tikzpicture}[scale=0.6,rotate=-90]
      \foreach \x in {1,...,5} {
        \node [v]  (v\x) at(0,\x){};
        \node [v]  (w\x) at (2,\x){};
        \draw (-.5,\x) node [above] {$v_\x$};
        \draw (w\x) node [below] {$w_\x$};
      }
      \draw (v1)--(v5);
      \draw(v1) [in=-120,out=120] to (v3);
      \draw(v2) [in=-120,out=120] to (v4);
      \draw(v3) [in=-120,out=120] to (v5);
      \draw(v1) [in=-120,out=120] to (v4);
      \draw(v2) [in=-120,out=120] to (v5);
      \draw(v1) [in=-120,out=120] to (v5);
      \foreach \x in {1,...,5} 
      \foreach \y in {1,...,5} {
        #1
      }
    \end{tikzpicture}
    }
    \Sfive{       \ifnum \x=\y  \draw (v\x)--(w\y);       \fi      }
    $\quad$
    \Sfive{       \ifnum \x=\y  \else      \draw (v\x)--(w\y);       \fi      }
    $\quad$
    \Sfive{       \ifnum \x>\y   \draw (v\x)--(w\y);       
    \else \ifnum\x=\y \draw (v\x)--(w\y);\fi\fi}
  \caption{$\K_5\mat \S_5$, $\K_5\antimat \S_5$, and $\K_5\tri \S_5$.}
  \label{fig:construction}
\end{figure}

	We will use the following lemma. Similar lemmas appeared in \cite[Lemma 2.8]{KwonO2014}, \cite[Proposition 6.2]{HlinenyKJS2016}, and \cite[Lemma 5.6]{KKOS2018}.
\begin{lemma}[Kwon and Oum {\cite[Lemma 6.5]{KO2018sca}}]\label{lem:antimattopath} 
Let $n$ be a positive integer.
\begin{enumerate}[(1)]
	\item $\S_{n}\tri\S_{n}$ is locally equivalent to $P_{2n}$.
	\item $\K_{n}\tri\S_{n}$ is locally equivalent to $P_{2n}$.
	\item If $n\ge 2$, then $\K_{n}\tri\K_{n}$ has a vertex-minor isomorphic to $P_{2n-2}$.
\end{enumerate}
\end{lemma}
	We also use the Ramsey-type result on bipartite graphs without twins.
			\begin{theorem}[Ding, Oporowski, Oxley, and 
  Vertigan~\cite{GuoliBJD1996}]\label{thm:largebipartite}
     For every positive integer $n$,
     there exists an integer $B(n)$ 
     such that for every bipartite graph $G$ with a bipartition $(S,T)$, 
     if no two vertices in $S$ have the same set of neighbors
     and      $\abs{S}\ge B(n)$,  
     then $S$ and $T$ have $n$-element subsets $S'$ and $T'$, respectively, such that 
     $G[S',T']$ is isomorphic to $\S_{n}\mat\S_{n}$, $\S_{n}\tri \S_{n}$, or
     $\S_{n}\antimat\S_{n}$. 
\end{theorem}
	
\begin{figure}
  \centering
  \tikzstyle{v}=[circle, draw, solid, fill=black, inner sep=0pt, minimum width=3pt]
  \begin{tikzpicture}[scale=0.7]
    \foreach \i in {1,2,3}{
      \foreach \x in {1,2,3,4}{
        \node [v]  (w\x\i) at (\x+\i*5,0){};
        \node [v]  (v\x\i) at(\x+\i*5,3.5){};
        
        \draw (\x+\i*5,3.8) node [above] {$v_\x^{(\i)}$};
        \draw (\x+\i*5,-2) node [below] {$w_\x^{(\i)}$};
      }
      \draw (v1\i)--(v4\i);
      \draw(v1\i) [in=150,out=30] to (v3\i);
      \draw(v2\i) [in=150,out=30] to (v4\i);
      \draw(v1\i) [in=150,out=30] to (v4\i);
      \foreach \x in {1,...,4} {
        \draw (v\x\i)--(w\x\i);
      }
    }
    \foreach \i in {1,2} {
      \foreach \j in {2,3} {
        \ifnum \i<\j {
        \foreach \x in {1,...,4}{
          \foreach \y in {1,...,4}{
            \draw (v\x\i)--(w\y\j);
            \draw (w\x\i) [in=-150,out=-30] to (w\y\j);
            }
          }
        }
        \fi
      }
    }
      
  \end{tikzpicture}
  \caption{The graph $(K_4\mat \S_4)^3_A$ for $A=
      \binom{0 ~ 1}{0~1}
      $.}
  \label{fig:construction2}
\end{figure}
For a positive integer $t$,
a $2\times 2$ matrix $A=(\begin{smallmatrix}a&b\\c&d\end{smallmatrix})$,
and a binary operator $\odot \in \{\mat, \antimat, \tri\}$ on two graphs of the same number of vertices,
we define $(G\odot H)^t _A$
as the graph on the disjoint union of $t$ copies of $G\odot H$
such that for all $1\le i<j\le t$,
\begin{enumerate}[(i)]
\item the $i$-th copy of $G$ is complete to
  the $j$-th copy of $G$ if $a=1$
  and anti-complete if $a=0$, 
\item the $i$-th copy of $G$ is complete to
  the $j$-th copy of $H$ if $b=1$,
  and anti-complete if $b=0$, 
\item the $i$-th copy of $H$ is complete to
  the $j$-th copy of $G$ if  $c=1$,
  and anti-complete if $c=0$, 
\item the $i$-th copy of $H$ is complete to
  the $j$-th copy of $H$ if  $d=1$,
  and anti-complete if $d=0$.
\end{enumerate}
See Figure~\ref{fig:construction2} for an illustration.

	\section{Lemma on three twins}\label{sec:rank1pair}
	
  In this section, we prove that if a graph has three pairwise twins, then one of them can be removed without decreasing its depth-$d$ rank-brittleness for $d\ge 2$.
  It holds for all $d\ge 2$ but we will only use it for $d=2$ later.
	
	\begin{lemma}\label{lem:reducetwin}
	Let $d\ge 2$ be an integer. 
        Let $v$, $w$, $z$ be vertices of a graph $G$ that are pairwise twins. 
	Then $\rbrit_d(G)=\rbrit_d(G-v)$.
	\end{lemma}
	\begin{proof}
	The inequality $\rbrit_d(G)\ge \rbrit_d(G-v)$ is trivial by definition.
    We will show that if $G-v$ has a $(k,d)$-decomposition $(T,\sigma)$,
    then $G$ also has a $(k,d)$-decomposition.

          Let $r$ be a node of $T$, called a \emph{root} of $T$,
          which has distance at most $d$ to every node of $T$.
	We may assume that 
	$r$ is not a leaf node.
	Let $a$ be the leaf of $T$ with $a=\sigma(w)$ and $b$ be the parent of $a$ in $T$, which is the unique neighbor of $a$ in $T$.
	We obtain a decomposition $(T_1, \sigma_1)$ of $G$ as follows:
	$T_1$ is the tree obtained from $T$ by adding a new node $a'$ adjacent to $b$, 
	and assign $\sigma_1(v):=a'$ and $\sigma_1(x):=\sigma(x)$ for all $x\in V(G)\setminus \{v\}$.
	We claim that $(T_1, \sigma_1)$ is a $(k, d)$-decomposition of $G$.
	Clearly, $T_1$ has radius at most $d$. So, it is sufficient to show that every internal node of $T_1$ has width at most $k$.
	For each internal node~$t$ of $T_1$, let $\mathcal{P}_t$ be the partition of $V(G)$ derived from the components of $T_1-t$ by $\sigma_1^{-1}$.

		For an internal node $t\neq b$ in $T_1$, the width of $t$ in $(T_1, \sigma_1)$ is the same as its width in the decomposition $(T, \sigma)$
		because $v$ and $w$ are twins of $G$ and $v$ and $w$ lie on the same part of $\mathcal{P}_t$.
		
		We claim that the width of $b$ in $(T_1, \sigma_1)$ is at most $k$.
		Let $\mathcal{P'}\subseteq \mathcal{P}_b$ and $A:=\bigcup_{X \in \mathcal{P'} } X$.
		In the bipartition 
		$(A, V(G)\setminus A)$, if $v$ is contained in a part together with $w$ or $z$, 
		then the bipartition obtained by removing~$v$  arises in the decomposition $(T, \sigma)$ as well.
		So, without loss of generality, we may assume that $w,z\in A$ and $v\in V(G)\setminus A$.
		But in this case, as $v,w,z$ are pairwise twins, 
		the bipartition obtained by exchanging $v$ and $w$ has the same cut-rank.
		As $w$ is a single-vertex part of $\mathcal{P}_b$,
		the bipartition $(A\setminus \{w\}, V(G)\setminus (A\setminus \{w\}))$ arises in the decomposition $(T, \sigma)$. 
		So, 
		\[\rho_G(A)=\rho_G((A\setminus \{w\})\cup \{v\})=\rho_{G-v}(A\setminus \{w\})\le k.\]

		We conclude that the width of every internal node of $T_1$ is at most $k$.
\end{proof}
	
	\section{Reducing to two cases}\label{sec:construction}
  We recall the definition of rank $k$-brittleness~\cite{KO2018sca}.
  The \emph{rank $k$-brittleness} of a graph $G$, denoted by $\rkbrit_k(G)$,
  is the minimum $\rho_G$-width of all partitions of $V(G)$ into 
  parts of size at most $k$.

	\begin{lemma}\label{lem:rkbrit}
		$\rbrit_2(G)\le \max (2, k, \rkbrit_k(G))$ for every positive integer $k$.
	\end{lemma}
	\begin{proof}
	Suppose that $(X_1,X_2,\ldots,X_m)$ is a partition of $V(G)$ whose $\rho_G$-width is $\rkbrit_k(G)$.
	We create a decomposition $(T, \sigma)$ of $G$ as follows.
	Let $r$ be the root of $T$, and let $r_1, r_2,\ldots, r_m$ be the children of $T$, 
	and each $r_i$ has exactly $\abs{X_i}$ leaves adjacent to $r_i$, and we assign $X_i$ to these leaves by $\sigma$.
	It is easy to see that each $r_i$ has width at most $k$, and 
	the root $r$ has width at most $\rkbrit_k(G)$.
	Thus, $\rbrit_2(G)\le \max (2, k, \rkbrit_k(G))$.
	\end{proof}
	Kwon and Oum~\cite{KO2018sca} proved the following.
        \begin{theorem}[Kwon and Oum~{\cite[Theorem 1.4]{KO2018sca}}]\label{thm:scattered}
          For every positive integer $n$,
          there exists $N$ such that
          every graph $G$ with $\beta^\rho_k(G)\ge N$
          contains a vertex-minor
          isomorphic to $nH$ for some connected graph $H$ on $k+1$ vertices.
        \end{theorem}
  Every large connected graph has a long induced path or a vertex of large degree. 
  \begin{proposition}[See Diestel~{\cite[Proposition 1.3.3]{Diestel2010}}]\label{prop:diestel}
    For integers $k>3$ and $\ell>0$, every connected graph on at least $(k-1)(k-2)^{\ell-2}/(k-3)$ vertices contains a vertex of degree at least $k$ or an induced path on $\ell$ vertices.
  \end{proposition}
  As a corollary we deduce the following. 
  Essentially its proof is almost identical to the proof of \cite[Theorem 1.6]{KO2018sca}.

        \begin{corollary}\label{cor:scattered}
          For all positive integers $k$ and  $n$, there exists $N=h(n,k)$ such that 
          every graph $G$ with depth-$2$ rank-brittleness at least $N$
          has a vertex-minor isomorphic to $n K_k$.
        \end{corollary}
        \begin{proof}
          We may assume that $k>3$ by increasing $k$ if necessary.
          Let 
          \[M:=\lceil (R(k-1;2)-1)(R(k-1;2)-2)^{2(k-1)-2}/(R(k-1;2)-3)\rceil .\]
          By Proposition~\ref{prop:diestel}, 
          every connected graph with at least $M$ vertices
          has a vertex of degree at least $R(k-1;2)$
          or an induced subgraph isomorphic to $P_{2(k-1)}$.
          By Theorem~\ref{thm:scattered},
          there exists $N$ such that
          $N\ge M$ and 
          every graph $G$ with $\beta^\rho_{M-1}(G)\ge N$
          contains $nH$  for some connected graph $H$ on $M$ vertices.
          By Lemma~\ref{lem:antimattopath}(2),
          if $H$ contains $P_{2k}$ as an induced subgraph, then $H$ contains $K_k$
          as a vertex-minor.
          If $H$ contains a vertex of degree at least $R(k-1;2)$, then it contains
          $K_{1,k-1}$ or $K_{k}$ as an induced subgraph.
          In all cases, $H$ contains $K_k$ as a vertex-minor and so
          $nH$ contains $nK_k$ as a vertex-minor.

          If $G$ has depth-$2$ rank-brittleness at least $N$,
          then by Lemma~\ref{lem:rkbrit},
          \[\beta^\rho_{M-1}(G)\ge N\]
          and therefore $G$ has a vertex-minor isomorphic to $nK_k$.
        \end{proof}

		\begin{proposition}\label{prop:main2}
		For every integer $n\ge 2$, there exists an integer $\sigma(n)$ such that 
                every graph $G$ of depth-$2$ rank-brittleness at least  $\sigma(n)$ contains a vertex-minor $G'$ satisfying one of the following.
                \begin{enumerate}[(i)]
                \item $V(G')=X_1^*\cup X_2^*\cup\cdots \cup X_n^*\cup Q^*$
                  for disjoint sets $X_1^*$, $X_2^*,\ldots, X_n^*,Q^*$ of $n$ vertices
                  such
                  that each $X_i^*$ is a clique in $G'$,
                  $G'[X_1^*\cup X_2^*\cup \cdots \cup X_n^*]$ is isomorphic to $nK_n$,
                  and
                  either $K_n\mat \S_n$ or $K_n\antimat \S_n$
                  is isomorphic to all $G'[X_i^*\cup Q^*]$.
                \item $V(G')=X_1^*\cup X_2^*\cup\cdots \cup X_n^*\cup Y_1^*
                  \cup Y_2^*\cup\cdots\cup Y_n^*$
                  for disjoint sets $X_1^*$, $X_2^*,\ldots, X_n^*,Y_1^*,Y_2^*,\ldots,
                  Y_n^*$ of $n$ vertices
                  such
                  that each $X_i^*$ is a clique in $G'$,
                  $G'[X_1^*\cup X_2^*\cup \cdots \cup X_n^*]$ is isomorphic to $nK_n$,
                  and
                  one of  $K_n\mat \S_n$,
                  $K_n\mat K_n$,  $K_n\antimat \S_n$, 
                  and $K_n\antimat K_n$
                  is isomorphic to 
                  all $G'[X_i^*\cup Y_i^*]$.
                \item $G'$ is isomorphic to $P_n$.
                \end{enumerate}
		\end{proposition}
		\begin{proof}
		If $G$ contains a component with at least $3$ vertices, then it contains a vertex-minor isomorphic to $P_3$.
		Thus, we may assume that $n\ge 4$.
		
		Let $B$ be the function defined in Theorem~\ref{thm:largebipartite}, 
                and $h$ be the function defined in Corollary~\ref{cor:scattered}.
                Let
                \begin{align*}
                  f_2(n)&:=R(2n;2),\\
		 f_1(n)&:= (2n)! n^{2n}+1 , \\
		 \sigma(n)&:=h(4f_1(n),2B(f_2(n))).
                \end{align*}

                We may assume that no proper vertex-minor of $G$
                has depth-$2$ rank-brittleness at least $\sigma(n)$.
		By Lemma~\ref{lem:reducetwin},
                every graph locally equivalent to $G$ has no three vertices that are pairwise twins.
		
                By Corollary~\ref{cor:scattered}, 
                $G$ has a vertex-minor isomorphic to 
                \[H:=(4f_1(n)) K_{2B(f_2(n))}.\]
                We may assume that
                $(4f_1(n)) K_{2B(f_2(n))}$ is an induced subgraph of $G$
                by applying local complementations.
		Let $C_1, C_2, \ldots, C_{4f_1(n)}$ be the set of connected components of $H$, and let $U:=V(G)\setminus V(H)$.
		
		Observe that 
		$G$ has no three vertices that are pairwise twins.
		It means that in each $C_i$, there are no three vertices that have the same neighborhood on $U$ in $G$, and 
		thus each $C_i$ contains a subset $S_i$ with  \[\abs{S_i}=\lceil {\abs{C_i}}/2\rceil=B( f_2(n) )\] that have pairwise distinct
                sets of neighbors on $U$.

		Now, we consider the bipartite graph $G[S_i, U]$ for each $i$. 
		In this bipartite graph, since vertices in $S_i$ have distinct neighborhoods on $U$ and $\abs{S_i}=B( f_2(n) )$, 
		by Theorem~\ref{thm:largebipartite}, 
		there exist $X_i'\subseteq S_i$ and $Y_i'\subseteq U$ such that
		$G[X_i', Y_i']$ is isomorphic to $\S_{f_2(n)}\mat\S_{f_2(n)}$, $\S_{f_2(n)}\tri \S_{f_2(n)}$, or
                $\S_{f_2(n)}\antimat\S_{f_2(n)}$
                for each $i\in\{1,2,\ldots,4f_1(n)\}$.

    As $f_2(n)=R(2n;2)$, by Ramsey's theorem, there exist $X_i\subseteq X_i'$ and $Y_i\subseteq Y_i'$ with $\abs{X_i}=\abs{Y_i}=2n$ 
	where
	\begin{itemize}
	\item $Y_i$ is a clique or an independent set in $G$, 
        \item $G[X_i, Y_i]$ is isomorphic to $\S_{2n}\mat\S_{2n}$, $\S_{2n}\tri \S_{2n}$, or
     $\S_{2n}\antimat\S_{2n}$.
  \end{itemize}
  This can be done by selecting $Y_i$ from $Y_i'$ by using Ramsey's theorem 
  and then selecting $X_i$ by using the relation between $X_i'$ and $Y_i'$.
	If $G[X_i, Y_i]$ is isomorphic to $\S_{2n}\tri \S_{2n}$ for some $i$, 
	then $G$ contains a vertex-minor isomorphic to a path on $4n-2\ge n$ vertices by Lemma~\ref{lem:antimattopath}.
	Thus, we may assume that 
	$G[X_i, Y_i]$ is isomorphic to $\S_{2n}\mat\S_{2n}$ or
        $\S_{2n}\antimat\S_{2n}$ for all $i\in\{1,2,\ldots,4f_1(n)\}$.
        So $G[X_i\cup Y_i]$ for each $i$ is isomorphic to
        $\K_{2n}\mat\S_{2n}$,
        $\K_{2n}\antimat\S_{2n}$,
        $\K_{2n}\mat\K_{2n}$, or
        $\K_{2n}\antimat\K_{2n}$.
        By the pigeonhole principle, we may assume that
        for all $i\in\{1,2,\ldots,f_1(n)\}$,
        all graphs $G[X_i\cup Y_i]$ are isomorphic to
        exactly one of 
        $\K_{2n}\mat\S_{2n}$,
        $\K_{2n}\antimat\S_{2n}$,
        $\K_{2n}\mat\K_{2n}$, and
        $\K_{2n}\antimat\K_{2n}$.

     	We now apply Theorem~\ref{thm:sunflower}, the sunflower lemma, 
        to sets  $Y_1, Y_2, \ldots, Y_{f_1(n)}$.
        As we choose $f_1(n)> (2n)! n^{2n}$, 
        $\{Y_1, Y_2, \ldots, Y_{f_1(n)}\}$ contains a sunflower $\mathcal{F}$ with $n+1$ petals.
        We may assume that $\mathcal F=\{Y_1,Y_2,\ldots,Y_{n+1}\}$.
        Let $Q$ be the core of $\mathcal{F}$, that is $\bigcap_{i=1}^{n+1} Y_i$.
        Note that  
        either $Q$ has at least $n+1$ vertices, 
        or $Y_i\setminus Q$ has at least $n$ vertices for all $i=1,2,\ldots,n+1$.
        We divide into two cases depending on the size of the core.
        
        First, suppose that $\abs{Q}\ge n+1$.
        Let $Q^*$ be a subset of $Q$ with $\abs{Q^*}=n$.
        For $i=1,2,\ldots,n$, let $X_i^*$ be the set of vertices in $X_i$
        paired with vertices in $Q^*\subseteq Y_i$ in the graph $G[X_i\cup Y_i]$.
        Then each $X_i^*$ is a clique and $G[X_1^*\cup X_2^*\cup\cdots\cup
        X_n^*]$ is isomorphic to $nK_n$.
        Let $w\in Q\setminus Q^*$.

        If all $G[X_i^*\cup Q^*]$ are isomorphic to $K_n\antimat K_n$,
        then
        $X_{n+1}$ has a vertex $v$ adjacent to all vertices in $Q^*$.
        In this case, we take $G':=G*v$. 
        Then for every $i\in \{1, 2, \ldots, n\}$, $G'[X_i^*\cup Q^*]$ is isomorphic to $K_n\antimat \S_n$.

        If all $G[X_i^*\cup Q^*]$ are isomorphic to $K_n\mat K_n$,
        then we take $G':=G*w$.
        Then all $G'[X_i^*\cup Q^*]$ are isomorphic to $K_n\mat \S_n$.
	
	If all $G[X_i^*\cup Q^*]$ are isomorphic to $K_n\mat \S_n$ or $K_n\antimat \S_n$, then
  we take $G':=G$. 
  
        We conclude that 
        $G$ has a vertex-minor $G'$ on $X_1^*\cup X_2^*\cup \cdots \cup X_n^*\cup Q^*$
        such that
        each $X_i^*$ is a clique in $G'$,
        $X_i^*$ is anti-complete to $X_j^*$  for all $i\neq j$ in~$G'$,
        and
        one of $K_n\mat \S_n$ or $K_n\antimat\S_n$
        is isomorphic to
        $G'[X_i^*\cup Q^*]$
        for all $i\in\{1,2,\ldots,n\}$.
        So, $G'$ provides a desired vertex-minor
        of the first type.

        Now it remains to consider the case that $\abs{Q}\le n$. Then
        for all $i\in\{1,2,\ldots,n\}$,
        $\abs{Y_i\setminus Q}\ge n$.
        For each $i=1,2,\ldots,n$,
        let $Y_i^*$ be a subset of $Y_i\setminus Q$ with $\abs{Y_i^*}=n$.
        For $i=1,2,\ldots,n$, let $X_i^*$ be the set of vertices in $X_i$
        paired with vertices in $Y_i^*$ in the graph $G[X_i\cup Y_i]$.
        Then we deduce that
        $G[X_1^*\cup X_2^*\cup \cdots \cup X_n^*]$ is isomorphic to $nK_n$
        and
        for all $i=1,2,\ldots,n$,
        all $G[X_i^*\cup Y_i^*]$ are isomorphic to
        exactly one of
        $K_n\mat \S_n$, $K_n\antimat\S_n$, $K_n\mat K_n$, or $K_n\antimat K_n$.
        and $X_1^*,X_2^*,\ldots,X_n^*,Y_1^*,Y_2^*,\ldots,Y_n^*$ are disjoint.
        So, $(X_1^*,X_2^*,\ldots,X_n^*)$ and $(Y_1^*,Y_2^*,\ldots,Y_n^*)$ provide a desired induced subgraph of the second type.
      \end{proof}

      In the rest, we will find a vertex-minor isomorphic to $P_n$ or $nT_{2,n}$ when a given graph satisfies (i) or (ii) of Proposition~\ref{prop:main2}.

      \subsection{The first case}\label{subsec:starmatching}
      \begin{lemma}\label{lem:first1}
        Let $n\ge 2$ be an integer.
        Let $G$ be a graph on $n^2+n$ vertices such that 
        $V(G)=X_1\cup X_2\cup\cdots \cup X_n\cup Q$
        for disjoint sets $X_1$, $X_2,\ldots, X_n,Q$ of $n$ vertices,
        each $X_i$ is a clique in $G$,
        $G[X_1\cup X_2\cup \cdots \cup X_n]$ is isomorphic to $nK_n$,
        and all $G[X_i\cup Q]$ are isomorphic to  $K_n\mat \S_n$.
        Then $G$ has a vertex-minor isomorphic to $P_{3n-1}$.
      \end{lemma}
      \begin{proof}
        Let $v_1$, $v_2$, $\ldots$, $v_n$ be an arbitrary enumeration of vertices in $Q$.
        For each $i\in\{1,2,\ldots,n-1\}$, there are two vertices $x_i$, $y_i$ in $X_i$ such that
        $N_G(x_i)\cap Q=\{v_i\}$, $N_G(y_i)\cap Q=\{v_{i+1}\}$.
        Let $x_n$ be the neighbor of $v_n$ in $X_n$.
        Then $v_1x_1y_1v_2x_2y_2v_3\cdots v_nx_n$ is an induced path on $3n-1$ vertices.
      \end{proof}
      \begin{figure}
        \begin{tikzpicture}[scale=.9]
          \tikzstyle{v}=[circle, draw, solid, fill=black, inner sep=0pt, minimum width=3pt]
          \node at(1,-0.9-1.8*5-0.4) [label=below:$G$]{};
          \fill [fill=blue,opacity=0.2,rounded corners] (1.2,-2.3) rectangle  (2.8,-0.9-1.8*5-0.4);
          \node at (2.4,-9.5) [label=above:$Q$]{};
          \foreach \x in {1,2,3,4,5}{
            \foreach \y in {1,2,3,4,5}{
              \node at (0,-1.8*\x-0.3*\y)[v] (v\x\y) {};
            }
            \foreach \y in {1,...,4}{
              \foreach \z in {\y,...,5}{
                \draw (v\x\y) [bend right] to (v\x\z); 
              }
            }
          }
          \foreach \x in {1,2,3,4,5}{
            \node at (2,-0.9-1.8*\x)[v] (w\x) {};
            \node at (v\x5) [v,label=left:$v_\x$]{};
          }
          \node at (w5) [label=right:$v$,v]{};
          \foreach \x in {1,2,3,4,5}{
            \foreach \y in {1,2,3,4,5}{
              \foreach \z in {1,2,3,4,5}{
                \ifnum \y = \z {
                  
                }
                \else {
                  \draw (v\x\y)--(w\z);
                }
                \fi
              }
            }
          }
          \begin{scope}[xshift=4cm]
            \node at(1,-0.9-1.8*5-0.4) [label=below:$G_1$]{};
            \fill [fill=blue,opacity=0.2,rounded corners] (1.2,-2.3) rectangle  (2.8,-0.9-1.8*5-0.4);
            \node at (2.2,-9.5) [label=above:$Q\setminus \{v\}$]{};            
          \foreach \x in {1,2,3,4}{
            \foreach \y in {1,2,3,4,5}{
              \node at (0,-1.8*\x-0.3*\y)[v] (v\x\y) {};
            }
            \foreach \y in {1,...,4}{
                \draw (v\x\y) [bend right] to (v\x5); 
            }
          }
          \foreach \x in {1,2,3,4}{
            \node at (2,-0.9-1.8*\x)[v,label=right:$w_\x$] (w\x) {};
            \node at (v\x5) [v,label=left:$v_\x$]{};
          }
          \foreach \x in {1,2,3,4}{
            \foreach \y in {1,2,3,4}{
                  \draw (v\x\y)--(w\y);
                  \draw (v\x5)--(w\y);
            }
          }
          \end{scope}
          \begin{scope}[xshift=8cm]
            \node at(1,-0.9-1.8*5-0.4) [label=below:$G_2$]{};
            \fill [fill=blue,opacity=0.2,rounded corners] (1.2,-2.3) rectangle  (2.8,-0.9-1.8*5-0.4);
            \node at (2.2,-9.5) [label=above:$Q\setminus \{v\}$]{};               
            \foreach \x in {1,2,3,4}{
              \foreach \y in {1,2,3,4,5}{
                \node at (0,-1.8*\x-0.3*\y)[v] (v\x\y) {};
              }
              \foreach \y in {1,...,4}{
                  \draw (v\x\y) [bend right] to (v\x5); 
              }
            }
            \foreach \x in {1,2,3,4}{
              \node at (2,-0.9-1.8*\x)[v,label=right:$w_\x$] (w\x) {};
              \node at (v\x5) [v,label=left:$v_\x$](v\x) {};
            }
            \foreach \x in {1,2,3,4}{
              \foreach \y in {1,2,3,4}{
                    \draw (v\x\y)--(w\y);
              }
            }
            \node at (v12) [v,label=left:$y_1$] (y1){};
            \node at (v23) [v,label=left:$y_2$] (y2){};
            \node at (v34) [v,label=left:$y_3$](y3){};

            \foreach \x in {1,2,3,4}{
              \node at (v\x\x) [v,label=left:$x_\x$](x\x){};
            }
            \draw [thick] 
            (w1)--(x1) [bend right] to (v1) [bend left] to (v12);
            \draw [thick](v12) --
            (w2)--(x2) [bend right] to (v2) [bend left] to (v23);
            \draw [thick](v23)--(w3)--(x3) [bend right] to (v3) [bend left] to (v34);
            \draw [thick] (v34)--            (w4)--(x4) [bend right] to (v4);
            \end{scope}          
        \end{tikzpicture}      
        \caption{Three graphs in the proof of Lemma~\ref{lem:first2} when $n=5$.}\label{fig:first2}
      \end{figure}
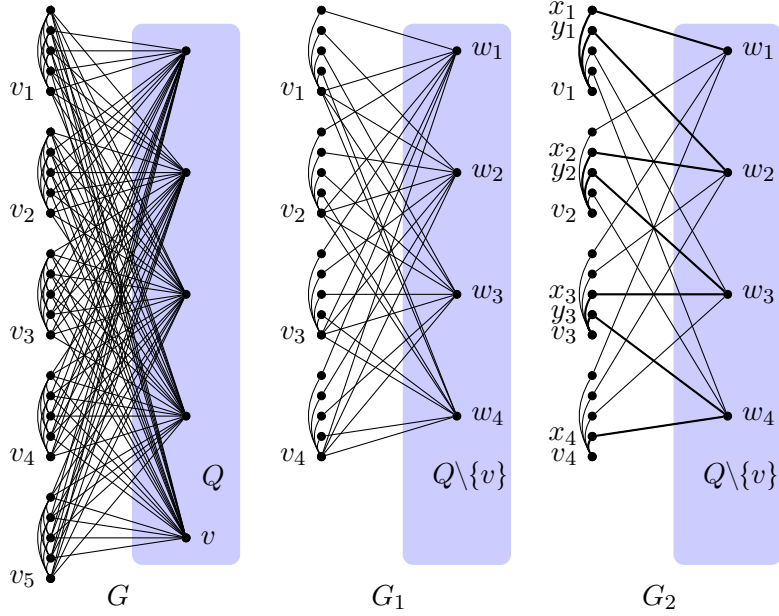
      \begin{lemma}\label{lem:first2}
      Let $n\ge 2$ be an integer.
        Let $G$ be a graph on $n^2+n$ vertices such that 
        $V(G)=X_1\cup X_2\cup\cdots \cup X_n\cup Q$
        for disjoint sets $X_1$, $X_2,\ldots, X_n,Q$ of $n$ vertices,
        each $X_i$ is a clique in $G$,
        $G[X_1\cup X_2\cup \cdots \cup X_n]$ is isomorphic to $nK_n$,
        and all $G[X_i\cup Q]$ are isomorphic to  $K_n\antimat \S_n$.
        Then $G$ has a vertex-minor isomorphic to $P_{4n-5}$.
      \end{lemma}
      \begin{proof}
        Let $v$ be a vertex in $Q$.
        Let $v_i$ be the vertex in $X_i$ non-adjacent to $v$.
        Let \[
          G_1=
        \begin{cases}
          G*v_1*v_2*\cdots*v_n-v-X_n &\text{if $n$ is even},\\
          G*v_1*v_2*\cdots*v_{n-1}-v-X_n&\text{if $n$ is odd.}
        \end{cases}
        \]
        Observe that every vertex in $X_i\setminus \{v_i\}$ has degree $2$ in $G_1$ and $N_{G_1}(v_i)=(X_i\setminus\{v_i\})\cup (Q\setminus \{v\})$ 
        for all $i\in \{1,2,\ldots,n-1\}$. See Figure~\ref{fig:first2} for an illustration.
        Let $G_2$ be the graph obtained from $G_1$ by applying local complementations at all vertices in $\bigcup_{i=1}^{n-1} (X_i\setminus\{v_i\})$.
        It is easy to see that $G_2$ is obtained from $G_1$ by deleting 
        all edges from $v_i$ to $Q\setminus \{v\}$ for all $i\le n-1$.
        Then $N_{G_2}(v_i)=X_i\setminus \{v_i\}$ and $G_2[(X_i\setminus\{v_i\})\cup (Q\setminus \{v\})]$ is isomorphic to $\S_{n-1}\mat \S_{n-1}$.
        Let $w_1$, $w_2$, $\ldots$, $w_{n-1}$ be an arbitrary enumeration of vertices in $Q\setminus\{v\}$.
        For each $i\in\{1,2,\ldots,n-2\}$, there are vertices $x_i,y_i\in X_i\setminus\{v_i\}$
        such that $N_{G_2}(x_i)\cap Q=\{w_i\}$ and $N_{G_2}(y_{i})\cap Q=\{w_{i+1}\}$.
        Let $x_{n-1}$ be the neighbor of $w_{n-1}$ in $X_{n-1}$.
        Then $w_ix_iv_iy_iw_{i+1}$ is an induced path and so 
        $w_1x_1v_1y_1w_2x_2v_2y_2\cdots w_{n-2}x_{n-2}v_{n-2}y_{n-2}w_{n-1}x_{n-1}v_{n-1}$
        is an induced path on $n-1+3(n-2)+2=4n-5$ vertices in $G_2$. 
        Thus, $G$ has a vertex-minor isomorphic to $P_{4n-5}$.
      \end{proof}

	\subsection{The second case}\label{subsec:dcm}

        We will use the product Ramsey theorem described below.
\begin{theorem}[{\cite[Theorem 11.5]{Trotter1992}}; See also \cite{GrahamRS1990}]\label{thm:productramsey}
Let $r, t$ be positive integers, and let $k_1, k_2, \ldots, k_t$ be nonnegative integers, and
let $m_1, m_2, \ldots, m_t$ be integers with $m_i\ge k_i$ for each $i\in \{1, 2, \ldots, t\}$.
Then there exists an integer $R=R_{prod}(r,t;k_1, k_2, \ldots, k_t; m_1, m_2, \ldots, m_t)$ such that
if $X_1, X_2, \ldots, X_t$ are sets with $\abs{X_i}\ge R$ for each $i\in \{1, 2, \ldots, t\}$, 
then for every function $f:{X_1 \choose k_1}\times {X_2 \choose k_2} \times \cdots \times {X_t \choose k_t} \rightarrow \{1, 2, \ldots, r\}$, 
there exist an element $\alpha\in \{1, 2, \ldots, r\}$ and subsets $Y_1, Y_2, \ldots, Y_t$ of $X_1, X_2, \ldots, X_t$, respectively, 
so that $\abs{Y_i}\ge m_i$ for each $i\in \{1, 2, \ldots, t\}$, and 
$f$ maps every element of ${Y_1\choose k_1}\times {Y_2\choose k_2}\times \cdots \times {Y_t\choose k_t}$ to $\alpha$.
\end{theorem}

\begin{lemma}\label{lem:dcm}
  For integers $s$ and $t$, there exist $M=f(s,t)$ and $N=g(s,t)$
  such that
  for $m\ge M$ and $n\ge N$, 
  if a graph $G$ has $2m$ disjoint $n$-vertex sets 
  $X_1$, $X_2,\ldots, X_m,Y_1,Y_2,\ldots,  Y_m$,
  each $X_i$ is a clique of $G$, 
  $G[X_1\cup X_2\cup \cdots\cup X_m]$ is isomorphic to $mK_n$,
  and
  one of $K_n\mat \S_n$,
  $K_n\mat K_n$,  $K_n\antimat \S_n$, 
  and $K_n\antimat K_n$
  is isomorphic to 
  all $G[X_i\cup Y_i]$,
  then
  there exist indices $1\le i_1<i_2<\cdots<i_s\le m$
  and subsets $X_1^*$, $X_2^*$, $\ldots$, $X_s^*$ of $X_{i_1}$, $X_{i_2}$, $\ldots$, $X_{i_s}$ respectively
  and subsets $Y_1^*$, $Y_2^*$, $\ldots$, $Y_s^*$ of $Y_{i_1}$, $Y_{i_2}$, $\ldots$,
  $\ldots$, $Y_{i_s}$ respectively
  such that the following hold.
  \begin{enumerate}[(i)]
  \item $\abs{X_i^*}=\abs{Y_i^*}=t$ for all $1\le i\le s$,
  \item   one of $K_t\mat \S_t$,
    $K_t\mat K_t$,  $K_t\antimat \S_t$, 
    and $K_t\antimat K_t$
    is isomorphic to 
    all $G[X_i^*\cup Y_i^*]$ for all $1\le i\le s$,
  \item $X_i^*$ is complete to $Y_j^*$ for all $i<j$
    or $X_i^*$ is anti-complete to $Y_j^*$ for all $i<j$,
  \item $Y_i^*$ is complete to $X_j^*$ for all $i<j$
    or $Y_i^*$ is anti-complete to $X_j^*$ for all $i<j$,
  \item $Y_i^*$ is complete to $Y_j^*$ for all $i<j$
    or $Y_i^*$ is anti-complete to $Y_j^*$ for all $i<j$.
  \end{enumerate}
  In other words, $G$ has an induced subgraph isomorphic to
  one of
  \[(K_t\mat K_t)^s_A,
  (K_t\mat \S_t)^s_A,
  (K_t\antimat K_t)^s_A, \text{ and }
  (K_t\antimat \S_t)^s_A\]
  for some $0$-$1$ matrix $A=
  (\begin{smallmatrix}
    0&b \\c&d 
  \end{smallmatrix})$.
\end{lemma}
\begin{proof}
  Let $m\ge M:=R(s;8)$ and  
  let  \[n\ge N:=R_{prod} (8^{\binom{m}{2}}, m; 1, 1, \ldots, 1; t,t,\ldots,t).\]
  The first step of the proof is to clean up edges between $X_i\cup Y_i$ and
  $X_j\cup Y_j$ for distinct $i$ and $j$.
  We consider a function that maps $(v_1,v_2,\ldots,v_m)$ for $v_i\in Y_i$
  to an edge-coloring
  of $K_m$ with colors on the edges $ij$
  based on the three possible adjacencies between
  a pair of $v_i$ and its unique neighbor or non-neighbor in $X_i$
  and a pair of $v_j$ and its unique neighbor or non-neighbor in $X_j$.
  Each edge of $K_m$ will receive one of $2^3$ colors and
  the range of this function has at most $8^{\binom{m}{2}}$ edge-colorings of $K_m$.
  By Theorem~\ref{thm:productramsey},
  there exist subsets $X_1'$, $X_2'$, $\ldots$, $ X_m'$, $Y_1'$, $Y_2'$, $\ldots$,  $Y_m'$ of
  $X_1$, $X_2,\ldots, X_m,Y_1,Y_2,\ldots,  Y_m$, respectively, such that
  \begin{enumerate}[(i)]
  \item $\abs{X_i'}=\abs{Y_i'}=t$,
  \item for each $i\neq j$, $X_i'$ is complete or anti-complete to $Y_j'$, and 
    $Y_i'$ is complete or anti-complete to $Y_j'$,
  \item 
    one of 
    $K_{t}\mat \S_{t}$,
    $K_{t}\mat K_{t}$,  $K_{t}\antimat \S_{t}$, 
    and $K_{t}\antimat K_{t}$
    is isomorphic to all $G[X_i'\cup Y_i']$.
  \end{enumerate}

  Now our next goal is to take a subset of $\{1,2,\ldots,m\}$ by using Ramsey's theorem.
  Let us color the edges $ij$ of $K_{m}$ ($i<j$) by the one of $8$ colors
  determined by the following:
  \begin{itemize}
  \item $X_i'$ is complete to $Y_j'$ or not.
  \item $Y_i'$ is complete to $X_j'$ or not.
  \item $Y_i'$ is complete to $Y_j'$ or not.
  \end{itemize}
  Then by Ramsey's theorem, there exists a subset $I$ of $\{1,2,\ldots,m\}$
  with $\abs{I}= s$  
  such that every edge of $K_m$ has the same color.
  Let $I=\{i_1,i_2,\ldots,i_s\}$ for $i_1<i_2<\cdots<i_s$
  and $X_j^*=X_{i_j}$, $Y_j^*=Y_{i_j}$ for $1\le j\le s$. This provides our conclusion.
\end{proof}

Now we will see that in many cases, we will have a vertex-minor
isomorphic to  $nT_{2,n}$.
\begin{lemma}\label{lem:obtainnt2}
  Let $n$ be a positive integer.
  \begin{enumerate}[(1)]
  \item $\K_{n+1}\mat\S_{n+1}$ contains a vertex-minor isomorphic to $T_{2,n}$.
  \item $\K_{n+2}\mat\K_{n+2}$ contains a vertex-minor isomorphic to $T_{2,n}$.
  \item $\K_{n+2}\antimat\S_{n+2}$ contains a vertex-minor isomorphic to $T_{2,n}$.
  \item $\K_{n+1}\antimat\K_{n+1}$ contains a vertex-minor isomorphic to $T_{2,n}$.
  \end{enumerate}
  Therefore, if $A=(
  \begin{smallmatrix}
    0 & 0\\
    0 & 0
  \end{smallmatrix})$, then 
  all of $(\K_{n+1}\mat \S_{n+1})^n_A$,
  $(\K_{n+2}\mat \K_{n+2})^n_A$,
  $(\K_{n+2}\antimat \S_{n+2})^n_A$,
  and 
  $(\K_{n+1}\antimat \K_{n+1})^n_A$
  have vertex-minors isomorphic to $nT_{2,n}$.
  
\end{lemma}
\begin{proof}
 (1) Let $V(\K_{n+1})=\{v_i:1\le i\le n+1\}$ and $V(\S_{n+1})=\{w_i:1\le i\le n+1\}$.
	The graph $(\K_{n+1}\mat\S_{n+1}-w_1)*v_1$ is isomorphic to $T_{2,n}$.
     
    (2) Let $\{v_i:1\le i\le n+2\}$ and $\{w_i:1\le i\le n+2\}$ be the vertex sets of two copies of $\K_{n+2}$. 
    The graph $((\K_{n+2}\mat\K_{n+2}-\{v_1, w_2\})*v_2*w_1)-\{w_1\}$ is isomorphic to $T_{2,n}$.

  (3) Let $V(\K_{n+2})=\{v_i:1\le i\le n+2\}$ and $V(\S_{n+2})=\{w_i:1\le i\le n+2\}$.
	The graph $((\K_{n+2}\antimat\S_{n+2}-\{w_1, v_2\})*v_1*w_2)-\{w_2\}$ is isomorphic to $T_{2,n}$.

 (4) Let $\{v_i:1\le i\le n+1\}$ and $\{w_i:1\le i\le n+1\}$ be the vertex sets of two copies of $\K_{n+1}$. 
  The graph $(\K_{n+1}\antimat\K_{n+1}-w_1)*v_1$ is isomorphic to the graph obtained from $\S_{n}\mat \S_{n}$ by adding a vertex $v_1$ adjacent to all other vertices.
    Thus, the graph $(\K_{n+1}\antimat\K_{n+1}-w_1)*v_1*v_2*\cdots *v_{n+1}$ is isomorphic to $T_{2,n}$.
  \end{proof}

In the following lemma, 
  we 
  show that
  if $A$ is not symmetric, then we obtain $P_n$ as a vertex-minor.
\begin{lemma}\label{lem:notsymmetric}
  Let $n\ge 2$ be an integer.
  If $A=(
  \begin{smallmatrix}
    0 & b \\c&d
  \end{smallmatrix})$
  is a $0$-$1$ matrix such that $b\neq c$,
  then 
  both $(\K_{1}\mat \K_{1})^n_A$
  and
  $(\K_{1}\antimat \K_{1})^{n+1}_A$
  have vertex-minors isomorphic to $P_{n}$.
\end{lemma}
\begin{proof}
  If $d=0$, then $(\K_{1}\mat \K_{1})^n_A$ is isomorphic to $K_n\tri \S_n$
  and $(K_{1}\antimat K_{1})^{n+1}_A$ contains an induced subgraph
  isomorphic to $K_n\tri \S_n$.
  If $d=1$, then  $(\K_{1}\mat \K_{1})^{n}_A$ is isomorphic to $K_n\tri \K_n$
  and $(\K_{1}\antimat \K_{1})^{n+1}_A$  contains an induced subgraph
  isomorphic to $K_n\tri \K_n$.
  By Lemma~\ref{lem:antimattopath}, there is a vertex-minor isomorphic to $P_{2n-2}$ in both cases. 
\end{proof}

\begin{lemma}\label{lem:simpledcm2}
  Let $n\ge 2$  be an integer.
  If $A=(
  \begin{smallmatrix}
    0 & 0\\
    0 &1 
  \end{smallmatrix}
  )$, then
  all of $(K_{n+2}\mat K_{n+2})^{n+1}_A$,
  $(K_{n+2}\mat \S_{n+2})^{n+1}_A$,
  $(K_{n+2}\antimat K_{n+2})^{n+1}_A$,
  and 
  $(K_{n+2}\antimat \S_{n+2})^{n+1}_A$
  have vertex-minors isomorphic to $nT_{2,n}$.
	\end{lemma}
	\begin{proof}
          Let $H$ be the one of   $K_{n+2}\mat K_{n+2}$,
          $K_{n+2}\mat \S_{n+2}$,
          $K_{n+2}\antimat K_{n+2}$,
          or
          $K_{n+2}\antimat \S_{n+2}$
          and let $G=H^{n+1}_A$.
          Let $X_1$, $Y_1$ be the sets of vertices of the first copy of $H$ in $G$
          where $X_1$ denotes the set of vertices in $K_{n+2}$
          and $Y_1$ denotes the other vertices.
          Let $w$ be a vertex in $Y_1$.
          
          Then $G*w$ contains an induced subgraph isomorphic to
          $(K_{n+2}\mat K_{n+2})^{n}_B$,
          $(K_{n+2}\mat \S_{n+2})^{n}_B$,
          $(K_{n+2}\antimat K_{n+2})^{n}_B$,
          or 
          $(K_{n+2}\antimat \S_{n+2})^{n}_B$
          for $B=(
          \begin{smallmatrix}
            0&0\\0&0
          \end{smallmatrix})$.
          By Lemma~\ref{lem:obtainnt2}, it has a vertex-minor isomorphic to
          $nT_{2,n}$.
	\end{proof}

	\begin{lemma}\label{lem:simpledcm3}
          Let $n\ge 2$  be an integer.
          If $A=(
          \begin{smallmatrix}
            0 & 1\\
            1 & d 
          \end{smallmatrix}
          )$ for some $d\in\{0,1\}$, then
          all of $(K_{n+2}\mat K_{n+2})^{n+2}_A$,
          $(K_{n+2}\mat \S_{n+2})^{n+2}_A$,
          $(K_{n+2}\antimat K_{n+2})^{n+2}_A$,
          and 
          $(K_{n+2}\antimat \S_{n+2})^{n+2}_A$
          have vertex-minors isomorphic to $nT_{2,n}$.
	\end{lemma}
  \begin{proof}
          Let $H$ be the one of   $K_{n+2}\mat K_{n+2}$,
          $K_{n+2}\mat \S_{n+2}$,
          $K_{n+2}\antimat K_{n+2}$,
          or
          $K_{n+2}\antimat \S_{n+2}$
          and let $G=H^{n+2}_A$.
          There exists an induced subgraph $G'$ of~$G$ and an edge $xy$ of~$G'$
          such that
          $G'- x- y$ is isomorphic to $H^{n+1}_A$,
          $x$ is complete to the bottom part of copies of $H$,
          and anti-complete to the top part of copies of $H$, 
          and 
          $y$ is complete to the top part of copies of $H$, 
          and is either complete or anti-complete to
          the bottom part of copies of $H$,
          because we can choose a vertex~$x$ from the top part of $H$
          and a neighbor $y$ is chosen from the bottom part in the same copy of $H$.
          Now, it is easy to see that $G'*x*y*x-x-y$ is isomorphic to
          one of 
 $(K_{n+2}\mat K_{n+2})^{n+1}_B$,
          $(K_{n+2}\mat \S_{n+2})^{n+1}_B$,
          $(K_{n+2}\antimat K_{n+2})^{n+1}_B$,
          and 
          $(K_{n+2}\antimat \S_{n+2})^{n+1}_B$
          for a matrix $B=
          (\begin{smallmatrix}
            0  & 0\\
            0 & d
          \end{smallmatrix})$.
          By Lemmas~\ref{lem:obtainnt2} and \ref{lem:simpledcm2},
          $G$ has a vertex-minor isomorphic to $nT_{2,n}$.
	\end{proof}

\section{Main proof}\label{sec:mainproof}

	We are ready to prove our main theorem, restated below.
	\begin{thm:main}
	A vertex-minor ideal $\mathcal{C}$ has bounded depth-$2$ rank-brittleness if and only if 
		\[ \{P_1, P_2, P_3, P_4, \ldots \}\nsubseteq \mathcal{C}, \]
		and
		\[ \{T_{2,1}, 2T_{2,2}, 3T_{2,3}, 4T_{2,4}, \ldots \}\nsubseteq \mathcal{C}. \]
	\end{thm:main}
	
        Before the proof, let us discuss why
        the two conditions in Theorem~\ref{thm:main}, 
      $\{P_1, P_2, P_3, \ldots \}\not\subseteq \mathcal C$		and
      $\{T_{2,1}, 2T_{2,2}, 3T_{2,3}, \ldots \}\not\subseteq \mathcal C$, are incomparable.
      First we sketch the proof showing 
      that no path contains $T_{2,3}$ as a vertex-minor.
      The tree $T_{2,3}$ is a tree having a vertex $v$ such that $T_{2,3}-v$ contains three components having linear rank-width $1$. 
      (The definition of linear rank-width will be discussed in Section~\ref{sec:application}.)
      It implies that it has linear rank-width at least $2$,
      by a characterization of linear rank-width on trees, see \cite{AdlerK2015, AdlerKK2017}.
      However, paths have linear rank-width $1$ and therefore
      no path contains $T_{2,3}$ as a vertex-minor.
      Thus, if $\mathcal C$ is the set of all vertex-minors of $P_n$ for all $n$, then $\mathcal C$ does not satisfy the first condition but satisfies the second condition.
      Secondly we claim that no $nT_{2,n}$ contains a long path as a vertex-minor.
      It is not difficult to see that $nT_{2,n}$ has depth-$3$ rank-brittleness at most $3$. 
      However, $\{P_1, P_2, P_3, \ldots \}$ has unbounded rank-depth~\cite{DKO2019}, and thus unbounded depth-$3$ rank-brittleness.
      So, if $\mathcal C$ is the set of all vertex-minors of $nT_{2,n}$ for all $n$,
      then $\mathcal C$ does not satisfy the second condition but satisfies the first condition.

      Now let us start the proof of Theorem~\ref{thm:main}.
      Our first lemma is to prove the forward implication.
      It is already known that $\{P_1, P_2, P_3, \ldots \}$ has unbounded rank-depth~\cite{DKO2019} and therefore it has unbounded depth-$2$
      rank-brittleness.
      Thus, to prove the forward implication,
      it is enough to show that  $\{T_{2,1}, 2T_{2,2}, 3T_{2,3}, 4T_{2,4}, \ldots \}$
      has unbounded depth-$2$ rank-brittleness.
	\begin{lemma}
	The class $\{T_{2,1}, 2T_{2,2}, 3T_{2,3}, 4T_{2,4}, \ldots \}$ has unbounded depth-$2$ rank-brittleness.
	\end{lemma}
	\begin{proof}
	We claim that $nT_{2,n}$ has depth-$2$ rank-brittleness at least $n/2$.
	Suppose that $nT_{2,n}$ admits a $(m,2)$-decomposition $(T, \sigma)$ with $m<n/2$.
        Then $T$ has a root $r$ from which every leaf is within distance at most $2$, and we may assume that $r$ is not a leaf.
	By subdividing an edge if necessary,
	we may further assume that no leaf is adjacent to $r$.
	
	Let $r_1$, $r_2$, $\ldots$, $r_m$ be the neighbors of $r$.
        We color each vertex $v$ of $nT_{2,n}$ by $i\in\{1,2,\ldots,m\}$
        if the component of $T-r$ containing $\sigma(v)$
        has $r_i$.
        An edge of $nT_{2,n}$ is \emph{colorful} if its ends have distinct colors.
	Let $C_1$, $C_2$, $\ldots$, $C_n$ be the components of $nT_{2,n}$.
	
	Suppose that a component $C_i$ is fully contained in $X_j$ for some $j$.
	Then, since $C_i$ contains an induced matching of size $n$, 
	the width of $r_j$ has to be at least $n$.
	This contradicts the assumption that $(T, \sigma)$ has width less than $n/2$. 
	Thus, we may assume that no component $C_i$ is fully contained in some $X_j$.
        So every component $C_i$ has a colorful edge
        and therefore $nT_{2,n}$ has a set $F$ of
        $n$ colorful edges in distinct components.

        Let $X$ be a subset of $\{1,2,\ldots,m\}$ chosen uniformly at random.
        A colorful edge of $nT_{2,n}$ is \emph{$X$-colorful} if
        one end has a color in $X$ and the other end has a color not in $X$.
        Then by the linearity of expectation,
        the expected number of $X$-colorful edges in $F$
        is $n/2$.
        This means that
        there exists $X$ such that there are at least $n/2$ $X$-colorful edges in distinct components of $nT_{2,n}$ and so 
        the width of $r$ is at least $n/2$, contradicting the assumption on $(T,\sigma)$.
      \end{proof}
	
      The following proposition proves the backward implication
      of Theorem~\ref{thm:main}.
        \begin{proposition}\label{prop:mainprop}
          For every integer $n\ge 2$, there exists an integer $N:=d(n)$ such that 
          every graph $G$ of depth-$2$ rank-brittleness at least  $N$ contains a vertex-minor isomorphic to $P_n$ or $nT_{2,n}$.
        \end{proposition}
        \begin{proof}
          Let $\sigma$ be the function defined in Proposition~\ref{prop:main2}
          and let $f$, $g$ be the functions defined in Lemma~\ref{lem:dcm}.
          Let $m:=\max(n, f(n+2,n+2),g(n+2,n+2))$.
          and let $d(n):=\sigma(m)$.
          By Proposition~\ref{prop:main2},
          $G$ has a vertex-minor $G'$ satisfying one of the following:
                \begin{enumerate}[(i)]
                \item $V(G')=X_1^*\cup X_2^*\cup\cdots \cup X_m^*\cup Q^*$
                  for disjoint sets $X_1^*$, $X_2^*,\ldots, X_m^*,Q^*$ of $m$ vertices
                  such
                  that each $X_i^*$ is a clique in $G'$,
                  $G'[X_1^*\cup X_2^*\cup \cdots \cup X_m^*]$ is isomorphic to $mK_m$,
                  and
                  either $K_m\mat \S_m$ or $K_m\antimat \S_m$
                  is isomorphic to all $G'[X_i^*\cup Q^*]$.
                \item $V(G')=X_1^*\cup X_2^*\cup\cdots \cup X_m^*\cup Y_1^*
                  \cup Y_2^*\cup\cdots\cup Y_m^*$
                  for disjoint sets $X_1^*$, $X_2^*,\ldots, X_m^*,Y_1^*,Y_2^*,\ldots,
                  Y_m^*$ of $m$ vertices
                  such
                  that each $X_i^*$ is a clique in $G'$,
                  $G'[X_1^*\cup X_2^*\cup \cdots \cup X_m^*]$ is isomorphic to $m K_m$,
                  and
                  one of  $K_m\mat \S_m$,
                  $K_m\mat K_m$,  $K_m\antimat \S_m$, 
                  and $K_m\antimat K_m$
                  is isomorphic to 
                  all $G'[X_i^*\cup Y_i^*]$.
                \item $G'$ is isomorphic to $P_m$.
                \end{enumerate}
                If (i) holds, then by Lemmas~\ref{lem:first1} and \ref{lem:first2}, $G'$ has a 
                vertex-minor isomorphic to $P_{3m-1}$ or $P_{4m-5}$.  So if (i) or (iii) holds, then 
                $G$ has a vertex-minor isomorphic to $P_n$.
                If (ii) holds, then 
                by Lemma~\ref{lem:dcm},
                $G'$ has an induced subgraph isomorphic to
                one of 
                \[(K_{n+2}\mat K_{n+2})^{n+2}_A,
                  (K_{n+2}\mat \S_{n+2})^{n+2}_A,
                  (K_{n+2}\antimat K_{n+2})^{n+2}_A, 
                  (K_{n+2}\antimat \S_{n+2})^{n+2}_A\]
                for some $0$-$1$ matrix $A=
                (\begin{smallmatrix}
                  0&b \\c&d 
                \end{smallmatrix})$.
                By Lemmas~\ref{lem:obtainnt2}, \ref{lem:notsymmetric}, \ref{lem:simpledcm2}, and \ref{lem:simpledcm3},
                $G'$ has a vertex-minor isomorphic to $P_n$ to $nT_{2,n}$.
      \end{proof}
	
		\section{Rank-depth and linear rank-width}\label{sec:application}
	
	By Theorem~\ref{thm:main}, 
for a fixed positive integer $n$, 
$nP_5$-vertex-minor free graphs have bounded 
depth-$2$ rank-brittleness, and thus have bounded rank-depth.
We will show that
they have bounded linear rank-width. 
Indeed, we will show that graphs of bounded rank-depth
have bounded linear rank-width.
This was also proved by 
Ganian, Hlin\v{e}n\'{y}, Ne{\v{s}}et{\v{r}}il, Obdr\v{z}\'{a}lek, 
and Ossona de Mendez~\cite[Proposition 3.4]{GHNOO2017} in terms of shrub-depth and linear clique-width, but our proof provides an explicit bound.

	First let us review the definition of linear rank-width \cite{Ganian2011,JKO2014,Oum2016}.
For a graph $G$, an ordering $(x_1, \ldots, x_n)$ of the vertex set $V(G)$ is called a \emph{linear layout} of $G$.  
	If $\abs{V(G)}\ge 2$, then the \emph{width} of a linear layout $(x_1,\ldots, x_n)$ of~$G$ is defined
as
$\displaystyle\max_{1\le i\le n-1}\cutrk_G(\{x_1,\ldots,x_i\})$,
and if $\abs{V(G)}=1$, then the width is defined to be $0$.
	The \emph{linear rank-width} of~$G$, denoted by $\lrw(G)$, is defined as the minimum width over all linear layouts of $G$. 
	It is easy to see that if $H$ is a vertex-minor of $G$, then $\lrw(H)\le \lrw(G)$.

	\begin{proposition}\label{prop:inequality}
          For a graph $G$, $\lrw(G)\le \rd (G)^2$.
	\end{proposition}
	\begin{proof}
	If $G$ has $1$ vertex, then $\lrw(G)=\rd(G)=0$. So, we may assume that $G$ has at least $2$ vertices.
	Let $k=\rd(G)$, and let $(T, \sigma)$ be a $(k,k)$-decomposition of $G$.
	Let $r$ be a node of $T$ within distance at most $k$ from every node of $T$.
	
	Let $v_1, v_2, \ldots, v_m$ be a DFS ordering of $T$.
        Let $n=\abs{V(G)}$.
        Let $w_1,w_2,\ldots,w_n$ be an ordering of the vertices of $G$
        such that for all $1\le i<j\le n$, $\sigma(i)$ appears before $\sigma(j)$
        in the DFS ordering $v_1,v_2,\ldots,v_m$ of $T$.
	We claim that $(w_1, w_2, \ldots, w_n)$ has width at most $k^2$.
	Let $i\in \{1, 2, \ldots, m\}$, $A_i:=\{v_1, \ldots, v_i\}$,  $B_i:=V(T)\setminus A_i$,
        $A_i'=\{v\in v(G): \sigma(v)\in A_i\}$,
        and $B_i'=\{v\in V(G):\sigma(v)\in B_i\}$.
	By the property of the depth-first search,
        $T$ has a path $P_i$ from $r$ consisting of nodes in $A_i$
        such that
        for each node $w$ in $B_i$,
        the first vertex in $A_i$ 
        in the path from $w$ to $r$ is on $P_i$. 

        As $T$ has radius at most $k$,
        we can take $P_i$ to have length at most $k-1$.

        For $w\in V(P_i)$, let $X_w$ be the set of all vertices $x$ of $G$
        mapped to a node $\sigma(x)$ in $B_i$
        such that $w$ is  the first vertex in $A_i$ 
        in the path from $\sigma(x)$ to $r$ is on $P_i$.

        Since $(T,\sigma)$ has width at most $k$,
        the cut-rank of $X_w$ is at most $k$.
        As $B_i'=\bigcup_{w\in V(P_i)} X_w$,
        we deduce that
        $\rho_G(B_i')\le \sum_{w\in V(P_i)} \rho_G(X_w)\le k^2$
        by the submodularity of the cut-rank function.
        This implies that the width of the linear layout is at most $k^2$.
	\end{proof}
	
	\begin{corollarylrw}
          For every positive integer $n$,
          graphs with no vertex-minors isomorphic to $nP_5$
          have 
          bounded depth-$2$ rank-brittleness,
          bounded rank-depth, and bounded linear rank-width.
	\end{corollarylrw}
	\begin{proof}
	Let $\mathcal{C}$ be the class of $nP_5$-vertex-minor free graphs. Then $P_{6n}\notin \mathcal{C}$ and $nT_{2,n}\notin \mathcal{C}$.
	Thus, by Theorem~\ref{thm:main}, $\mathcal{C}$ has bounded depth-$2$ rank-brittleness, and thus bounded rank-depth.
	By Proposition~\ref{prop:inequality}, it also has bounded linear rank-width.
	\end{proof}
	
		\providecommand{\bysame}{\leavevmode\hbox to3em{\hrulefill}\thinspace}
\providecommand{\MR}{\relax\ifhmode\unskip\space\fi MR }
\providecommand{\MRhref}[2]{%
  \href{http://www.ams.org/mathscinet-getitem?mr=#1}{#2}
}
\providecommand{\href}[2]{#2}

\end{document}